\documentclass[final,10pt,a4paper]{article}

\usepackage[notref,notcite]{showkeys}
\usepackage{fullpage}
\usepackage[dvips]{graphicx}
\usepackage{psfrag}
\usepackage{amsmath}
\usepackage{fancybox}
\usepackage{amsfonts}
\usepackage{amsthm}
\usepackage{amssymb}
\usepackage{enumerate}
\theoremstyle{plain}
   \newtheorem{theorem}{Theorem}[section]
   
   \newtheorem{lemma}[theorem]{Lemma}
   \newtheorem{proposition}[theorem]{Proposition}
\theoremstyle{definition}
   
   \newtheorem{remark}[theorem]{Remark}
\theoremstyle{remark}

\usepackage[%
   pdfpagemode=UseNone,
   bookmarks=true,
   pagebackref=false, 
   colorlinks,
   linkcolor=blue,
   anchorcolor=blue,
   citecolor=blue,
   filecolor=blue,
   urlcolor=blue
   ]{hyperref}

\usepackage{color}

\newcommand{\Taustar}{{\mathcal T_*}}

\newcommand{\h}{H}

\DeclareMathOperator{\osci}{osc}

\newcommand{\osc}[2]{\osci_\Tau(#1;#2)}

\newcommand{\gosck}[1]{\osci_{\Tau_k}(#1)}

\usepackage{xspace}

\newcommand{\step}[1]{\noindent\raisebox{1.5pt}[10pt][0pt]{\tiny\framebox{$#1$}}\xspace}
\newcommand{\definedas}{:=}
\newcommand{\asdefined}{=:}
\newcommand{\D}{\displaystyle}

\newcommand{\normT}[1]{\left\| {#1} \right\|_{T}}
\newcommand{\normNT}[1]{\left\| {#1} \right\|_{\omega_\Tau(T)}}

\newcommand{\normS}[1]{\left\| {#1} \right\|_{S}}
\newcommand{\normbT}[1]{\left\| {#1} \right\|_{\partial T}}

\DeclareMathOperator{\diam}{diam}

\newcommand{\est}[2]{\eta_\Tau(#1;#2)}
\newcommand{\gest}[1]{\eta_\Tau(#1)}

\newcommand\FF{\mathcal F}
\newcommand\JJ{\mathcal J}

\newcommand\NN{\mathbb N}

\newcommand\AAA{\mathcal A}

\newcommand\MM{\mathcal M}
\newcommand\Tau{\mathcal T}
\newcommand\PP{\mathcal P}
\newcommand\TT{\mathbb T}
\newcommand\VV{\mathbb V}
\newcommand\WW{\mathbb W}

\newcommand\RRR{\mathbf R}

\newcommand{\RR}{\mathbb R}
\newcommand{\nref}{n}

\begin{document}

\title{Quasi-optimal convergence rate of an\\ AFEM for quasi-linear problems}

\author{Eduardo M. Garau \and Pedro Morin \and Carlos Zuppa}

\maketitle

\begin{abstract}
We prove the quasi-optimal convergence of a standard adaptive finite element method (AFEM) for nonlinear elliptic second-order equations of monotone type. The adaptive algorithm is based on residual-type a posteriori error estimators and D\"orfler's strategy is assumed for marking. We first prove a contraction property for a suitable definition of total error, which is equivalent to the total error as defined by Casc\'on et al.~\cite{CKNS-quasi-opt}, and implies linear convergence of the algorithm. Secondly, we use this contraction to derive the optimal cardinality of the AFEM.
\end{abstract}

\begin{quote}\small
\end{quote}

\section{Introduction}

The main goal of this article is the study of convergence and optimality properties of an adaptive finite element method (AFEM) for 
quasi-linear elliptic partial differential equations over a polygonal/polyhedral domain $\Omega\subset \RR^d$ ($d=2,3$)  having the form
\begin{equation}\label{E:cons law}
\left\{
\begin{aligned}
-\nabla\cdot \big[\alpha(\,\cdot\,,|\nabla u|^2)\nabla u\big]&= f\qquad & & \text{in}\,\Omega\\
u&= 0\qquad & &\text{on}\,\partial \Omega,
\end{aligned}
\right.
\end{equation}
where $\alpha:\Omega\times \RR_+\to \RR_+$ is a bounded positive function whose precise properties will be stated in Section~\ref{S:setting} below, and $f\in L^2(\Omega)$ is given. This kind of problems arise in many practical situations, for example, in shock-free airfoil design, seepage through coarse grained porous media, and in some glaciological problems~\cite{Chow}.

AFEMs are an effective tool for making an efficient use of the computational resources, and for certain problems, it is even indispensable to their numerical resolvability.  
The ultimate goal of AFEMs is to equidistribute the error and the computational effort obtaining a sequence of meshes with optimal complexity. Adaptive methods are based on a posteriori error estimators, that are computable quantities depending on the discrete solution and data, and indicate a distribution of the error. 
A quite popular, natural adaptive version of classical finite element methods consists of the loop
\begin{equation}\label{E:adaptive loop}
 \textsc{Solve $\to$ Estimate $\to$ Mark $\to$ Refine},
\end{equation}
that is: solve for the finite element solution on the current grid, compute the a~posteriori error estimator, mark with its help elements to be subdivided, and refine the current grid into a new, finer one. 
 
A general result of convergence for linear problems has been obtained by Morin, Siebert and Veeser~\cite{MSV-convergence}, where very general conditions on the linear problems and the adaptive methods that guarantee convergence are stated. Following these ideas a (plain) convergence result for elliptic eigenvalue problems has been proved in~\cite{GMZ-08}. 
On the other hand, optimality of adaptive methods using \textit{D{\"o}rfler's marking strategy}\cite{Dorfler} for linear elliptic problems has been stated by Stevenson\cite{Stevenson} and Casc\'on, Kreuzer, Nochetto and Siebert\cite{CKNS-quasi-opt}. Linear convergence of an AFEM for elliptic eigenvalue problems has been proved in~\cite{Giani}, and optimality results can be found in~\cite{GM-09, xu-etal}. For a summary of convergence and optimality results of AFEM we refer the  reader to the survey~\cite{NSV-survey} and the references therein. We restrict ourselves to those references strictly related to our work.

Well-posedness and finite element error estimates for problem~\eqref{E:cons law} have been stated in~\cite{Chow}. A posteriori error estimators for nonconforming approximations have been developed in~\cite{Padra}. 
Linear convergence of an AFEM for the $\varphi$-Laplacian problem in a context of Sobolev-Orlicz spaces has been established in~\cite{diening-kreuzer}.
Recently, the (plain) convergence of an adaptive \emph{inexact} FEM for problem~\eqref{E:cons law} has been proved in~\cite{GMZ-kachanov}, where \emph{only} a discrete linear system is solved before each adaptive refinement.

In this article we consider a standard adaptive loop of the form~\eqref{E:adaptive loop} based on classical residual-type a posteriori error estimators, where the Galerkin discretization for problem~\eqref{E:cons law} is considered. We use the D\"orfler's strategy for marking and assume a minimal bisection refinement. The goal of this paper is to prove the optimal complexity of this AFEM by stating two main results. The first one establishes the convergence of the adaptive loop through a contraction property. More precisely, we will prove the following

\begin{theorem}[Contraction property]
Let $u$ be the weak solution of problem~\eqref{E:cons law} and let $\{U_k\}_{k\in\NN_0}$ be the sequence of discrete solutions computed through the adaptive algorithm described in Section~\ref{S:convergence}. Then, there exist constants $0<\rho<1$ and $\mu>0$ such that
$$[\FF(U_{k+1})-\FF(u)]+\mu \eta_{k+1}^{2} \leq \rho^2([\FF(U_k)-\FF(u)]+\mu \eta _k^{2}),\qquad\forall\,k\in\NN_0,$$
where $[\FF(U_k)-\FF(u)]$ is a notion equivalent to the energy error and $\eta_k$ denotes the global a posteriori error estimator in the mesh corresponding to the step $k$ of the iterative process.
\end{theorem}

The second main result shows that, if the solution of the nonlinear problem~\eqref{E:cons law} can be ideally approximated with adaptive meshes at a rate $(DOFs)^{-s}$, then the adaptive algorithm generates a sequence of meshes and discrete solutions which converge with this rate. Specifically, we will prove the following

\begin{theorem}[Quasi-optimal convergence rate]
Assume that the solution $u$ of problem~\eqref{E:cons law} belongs to $\mathbb{A}_{s}$.\footnote{Roughly speaking, $u\in\mathbb{A}_s$ if $u$ can be approximated with adaptive meshes with a rate $(DOFs)^{-s}$ (cf.~\eqref{E:As} in Section~\ref{S:optimality}).} Let $\{\Tau_k\}_{k\in\NN_0}$ and $\{U_k\}_{k\in\NN_0}$ denote the sequence of meshes and discrete solutions computed through the adaptive algorithm described in Section~\ref{S:convergence}, respectively. If the marking parameter $\theta$ in D\"orfler's criterion is small enough (cf.~\eqref{E:dorfler} and~\eqref{E:restriccion theta}), then
\begin{equation*}
\big[\|\nabla(U_k-u)\|_\Omega^2+\osci_{\Tau_k}^2(U_k)\big]^{\frac12}=\mathcal{O}\left( (\#\Tau_k-\#\Tau_0)^{-s}\right),\qquad\forall\,k\in\NN.
\end{equation*}
The left-hand side is called \emph{total error} and consists of the energy error plus an oscillation term.
\end{theorem}

Basically, we follow the steps presented in~\cite{CKNS-quasi-opt} for linear elliptic problems. However, 
due to the \emph{nonlinearity} of problem~\eqref{E:cons law} the generalization of the mentioned results is not obvious. In particular, the \emph{Galerkin orthogonality property (Pythagoras)}
\begin{equation}\label{E:Galerkin orthogonality}
\|\nabla (U-u)\|_\Omega^2+\|\nabla (U-V)\|_\Omega^2=\|\nabla (V-u)\|_\Omega^2,
\end{equation}
where $U$ is a discrete solution and $V$ is a discrete test function, is used for linear elliptic problems in order to prove the contraction property and a generalized Cea's Lemma (the quasi-optimality of the total error), and does not hold when we consider problem~\eqref{E:cons law}. To overcome this difficulty we resort to ideas from~\cite{diening-kreuzer}, replacing~\eqref{E:Galerkin orthogonality} by the trivial equality
$$[\FF(U)-\FF(u)]+[\FF(V)-\FF(U)]=[\FF(V)-\FF(u)],$$
where each term in brackets is equivalent to the corresponding term in~\eqref{E:Galerkin orthogonality} (cf. Theorem~\ref{T:equivalence for error} below), and thus establish some kind of \emph{quasi-orthogonality relationship} for the energy error (cf. Lemma~\ref{L:quasi-ortho mesh}) which is sufficient to prove the quasi-optimality of the total error (cf. Lemma~\ref{L:generalized Cea}).

Additionally, it is necessary to study the behavior of the error estimators and oscillation terms when refining. In order to do that, we need to show that certain quantity, which measures the difference of error estimators and oscillation terms between two discrete functions (cf.~\eqref{E:gtau}), is bounded by the energy of the difference between these functions (see Lemma~\ref{L:bound for gtau} in Section~\ref{S:reducciones nolineal}). This result can be proved with usual techniques for linear elliptic problems using inverse inequalities and trace theorems, but the generalization of this result to nonlinear problems requires some new technical results. We establish suitable hypotheses on the main coefficient $\alpha$ of problem~\eqref{E:cons law} to be able to prove the mentioned estimation for the nonlinear problems that we study in this article.

This paper is organized as follows. In Section~\ref{S:setting} we present specifically the problem that we study and some of its properties. In Section~\ref{S:a posteriori error estimates}, we present a posteriori error estimations. In Section~\ref{S:convergence} we state the adaptive loop that we use for the approximation of problem~\eqref{E:cons law} and we prove its linear convergence through a contraction property. Finally, the last two sections of the article are devoted to prove that the AFEM converges with quasi-optimal rate.

\section{Setting}\label{S:setting}

Let $\Omega\subset\RR^d$ be a bounded polygonal ($d = 2$) or polyhedral ($d = 3$) domain with Lipschitz boundary. A weak formulation of~\eqref{E:cons law} consists in finding $u\in H^1_0(\Omega)$ such that
\begin{equation}\label{E:cont prob}
a(u;u,v)=L(v),\qquad\forall\,v\in H^1_0(\Omega),
\end{equation}
where
\begin{equation*}
a(w;u,v)=\int_\Omega \alpha(\,\cdot\,,|\nabla w|^2)\nabla u\cdot \nabla v,\qquad\forall\,w,u,v\in H^1_0(\Omega),
\end{equation*}
and
\begin{equation*}
L(v)=\int_\Omega fv,\qquad\forall\,v\in H^1_0(\Omega).
\end{equation*}
In order to make this presentation clearer, we define $\beta:\Omega\times \RR_+\to\RR_+$ by 
\begin{equation*}
\beta(x,t)\definedas\frac12\int_0^{t^2}\alpha(x,r)~dr,
\end{equation*}
and note that from Leibniz's rule the derivative of $\beta$ as a function of its second variable satisfies
\begin{equation*}
D_2\beta(x,t)\definedas\frac{\partial\beta}{\partial t}(x,t)=t\alpha(x,t^2).
\end{equation*}
We require that $\alpha$ is $\mathcal{C}^1$ as a function of its second variable and there exist positive constants $c_a$ and $C_a$ such that
\begin{equation}\label{E:main cond} 
c_a\le \frac{\partial^2\beta}{\partial t^2}(x,t)=\alpha(x,t^2)+2t^2D_2\alpha(x,t^2)\le C_a,\qquad\forall\,x\in\Omega,\,t>0.
\end{equation}
Since $\alpha(x,t^2) = \frac{D_2\beta(x,t)-D_2\beta(x,0)}{t}=\frac{\partial^2 \beta}{\partial t^2}(x,r)$, for some $0<r<t$ the last assumption yields
\begin{equation}\label{E:alpha bounded}
c_a\le\alpha(x,t)\le C_a,\qquad\forall\,x\in\Omega,\,t>0.
\end{equation}
It is easy to check that the form $a$ is linear and symmetric in its second and third variable. Additionally, from~\eqref{E:alpha bounded} it follows that $a$ is bounded,
\begin{equation}\label{E:a bounded}
|a(w;u,v)|\le C_a \|\nabla u\|_\Omega\|\nabla v\|_\Omega,\qquad\forall\,w,u,v\in H^1_0(\Omega),
\end{equation}
and coercive,
\begin{equation*}
 c_a\|\nabla u\|_\Omega^2\le a(w;u,u),\qquad\forall\,w,u\in H^1_0(\Omega).
\end{equation*}

Now, we sketch the proof that~\eqref{E:main cond} is sufficient to guarantee the well-posedness of problem~\eqref{E:cont prob}. Let $\gamma:\Omega\times\RR^d\to\RR_+$ be given by
\begin{equation*}
\gamma(x,\xi)\definedas\beta(x,|\xi|)=\frac12\int_0^{|\xi|^2}\alpha(x,r)~dr,
\end{equation*}
and note that if $\nabla_2\gamma$ denotes the gradient of $\gamma$ as a function of its second variable, then
\begin{equation}\label{E:gradiente de gamma}
\nabla_2\gamma(x,\xi)=\alpha(x,|\xi|^2)\xi,\qquad\forall\,x\in\Omega,\,\xi\in\RR^d.
\end{equation}
Condition~\eqref{E:main cond} means that $D_2\beta$ is Lipschitz and strongly monotone as a function of its second variable and it can be seen that $\nabla_2 \gamma$ so is~\cite{Zeidler}.

If $A:H^1_0(\Omega)\to H^{-1}(\Omega)$ is the operator given by
\begin{equation*}
\langle Au,v\rangle \definedas a(u;u,v),\qquad\forall\,u,v\in H^1_0(\Omega),
\end{equation*}
problem~\eqref{E:cont prob} is equivalent to the equation
$$Au=L,$$
where $L\in H^{-1}(\Omega)$ is given. It is easy to check that the properties of $\nabla_2 \gamma$ are inherited by $A$, i.e., $A$ is Lipschitz and strongly monotone. More precisely, there exist positive constants $C_A$ and $c_A$ such that
\begin{equation}\label{E:A Lipschitz}
\|Au-Av\|_{H^{-1}(\Omega)}\le C_A \|\nabla (u -v)\|_\Omega,\qquad \forall\,u,v\in H^1_0(\Omega),
\end{equation}
and
\begin{equation}\label{E:A strongly mon}
\langle Au-Av,u-v\rangle\ge c_A\|\nabla (u -v)\|_\Omega^2,\qquad \forall\,u,v\in H^1_0(\Omega).
\end{equation}
As a consequence of~\eqref{E:A Lipschitz} and~\eqref{E:A strongly mon}, problem~\eqref{E:cont prob} has a unique stable solution~\cite{Zarantonello,Zeidler}, which will be denoted throughout this article by $u$.

\section{Discrete solutions and a posteriori error analysis}\label{S:a posteriori error estimates}

\subsection{Discretization}

In order to define discrete approximations to problem~\eqref{E:cont prob} we will consider \emph{triangulations} of the domain $\Omega$. Let $\Tau_0$ be an
initial conforming triangulation of $\Omega$, that is, a partition
of $\Omega$ into $d$-simplices such that if two elements intersect,
they do so at a full vertex/edge/face of both elements. Let us also assume that the initial mesh $\Tau_0$ is labeled satisfying condition~(b) of Section~4 in Ref.~\cite{Stevenson-refine}. Let $\TT$ denote the set of all conforming triangulations of $\Omega$ obtained from $\Tau_0$ by refinement using the
bisection procedure described by Stevenson~\cite{Stevenson-refine}, which coincides,  (after some re-labeling) with the \textit{newest vertex} bisection procedure in two dimensions and the  Kossaczk\'y's procedure in three
dimensions~\cite{Alberta}.

Due to the processes of refinement used, the family $\TT$ is shape regular, i.e., 
$$\sup_{\Tau\in\TT}\,  \sup_{T \in \Tau} \frac{\diam(T)}{\rho_T} \asdefined \kappa_\TT <\infty,$$
where $\diam(T)$ is the diameter of $T$, and $\rho_T$ is the radius of the largest ball contained in it. Throughout this article, we only consider meshes $\Tau$ that belong to the family $\TT$, so the shape regularity of all of them is bounded by the uniform constant $\kappa_\TT$ which only depends on the initial triangulation $\Tau_0$~\cite{Alberta}. Also, the diameter of any element $T\in\Tau$ is equivalent to the local mesh-size $H_T\definedas|T|^{1/d}$, which in turn defines the global mesh-size $H_\Tau \definedas \D\max_{T\in\Tau} H_T$. Also, the complexity of the refinement can be controlled, as described in Lemma~\ref{L:stevenson} below.

Hereafter, we denote the subset of $\Tau$ consisting of neighbors of $T$ by $\mathcal{N}_\Tau(T)$ and the union of $T$ and its neighbors in $\Tau$ by $\omega_\Tau(T)$. More precisely,
\[
\mathcal{N}_\Tau(T) \definedas \{ T' \in \Tau \mid\, T' \cap T \neq \emptyset \},
\qquad
\omega_\Tau(T) \definedas \bigcup_{T' \in \mathcal{N}_\Tau(T)} T' .
\]

For the discretization we consider the Lagrange finite element spaces consisting of continuous functions vanishing on $\partial \Omega$ which are piecewise linear over a mesh $\Tau\in\TT$, i.e.,
\begin{equation}\label{E:V_Tau}
\VV_{\Tau}\definedas\{V\in H^1_0(\Omega)\mid\quad V_{|_T}\in \PP_1(T),\quad \forall~T\in\Tau\}.
\end{equation}
The discrete problem associated to~\eqref{E:cont prob} consists in finding $U\in \VV_\Tau$ such that
\begin{equation}\label{E:disc prob}
a(U;U,V)=L(V),\qquad\forall\,V\in\VV_\Tau.
\end{equation}
Note that the discrete problem~\eqref{E:disc prob} has a unique solution because $A_{|_{\VV_\Tau}}$ is Lipschitz and strongly monotone (cf.~\eqref{E:A Lipschitz}--\eqref{E:A strongly mon}).

At this point, it is important to remark that the discrete problem~\eqref{E:disc prob} is also \emph{nonlinear}, and for our analysis we will assume that it can be solved exactly in every mesh $\Tau\in\TT$. However, this assumption is usual even though in practice, even for discrete \emph{linear} problems, we \emph{compute} only approximations to the solution of discrete problems. The optimality of inexact methods has been studied for linear problems in~\cite{Stevenson,MZ-optimality-inexact}, and a generalization to nonlinear problems is subject of future work.


\subsection{A posteriori error estimators}

In this section we present the \emph{a posteriori error estimators} for the discrete approximation~\eqref{E:disc prob} of problem~\eqref{E:cont prob} and state results showing their \emph{reliability} and \emph{efficiency}. These estimations will be useful in order to prove the optimality of the AFEM in Section~\ref{S:optimality}.

The \emph{residual} of $V\in\VV_\Tau$ is given by
$$\langle\RRR(V),v\rangle\definedas a(V;V,v)-L(v),\qquad\forall\,v\in H^1_0(\Omega).$$
Integrating by parts on each $T\in\Tau$ we have that
\begin{equation*}
\langle\RRR(V),v\rangle=\sum_{T\in\Tau} \left(\int_T R_\Tau(V)v +\int_{\partial T} J_\Tau(V)v\right),\qquad\forall\,v\in H^1_0(\Omega),
\end{equation*}
where $R_\Tau(V)$ denotes the \emph{element residual} given by 
\begin{equation}\label{E:element-residual}
{R_\Tau(V)}_{|_{T}}\definedas -\nabla \cdot [\alpha(\,\cdot \,,|\nabla V|^2)\nabla V]-f,\qquad\forall\,T\in\Tau,
\end{equation}
and $J_\Tau(V)$ the \emph{jump residual} given by
\begin{equation}\label{E:jump-residual}
{J_\Tau(V)}_{|_{S}}\definedas\frac12\left[(\alpha(\,\cdot \,,|\nabla V|^2)\nabla V)_{|_{T_1}}\cdot\vec{n}_1+(\alpha(\,\cdot \,,|\nabla V|^2)\nabla V)_{|_{T_2}}\cdot\vec{n}_2 \right],
\end{equation}
for each interior side $S$, and ${J_\Tau(V)}_{|_{S}}\definedas 0$, if $S$ is a side lying on the boundary of $\Omega$. Here, $T_1$ and $T_2$ denote the elements of $\Tau$ sharing $S$, and $\vec{n}_1$ and $\vec{n}_1$ are the outward unit normals of $T_1$ and $T_2$ on $S$, respectively.

We define the \emph{local a posteriori error estimator} $\est{V}{T}$ of $V\in\VV_\Tau$ by
\begin{equation}\label{E:local estimator}
\eta_\Tau^2(V;T)\definedas \h_T^2\normT{R_\Tau(V)}^2 + \h_T\normbT{J_\Tau(V)}^2, \quad\forall\,T\in\Tau,
\end{equation}
and the \emph{global error estimator} $\gest{V}$ by
$$
\eta_\Tau^2(V)\definedas \sum_{T\in\Tau} \eta_\Tau^2(V;T).
$$
In general, if $\Xi\subset\Tau$ we denote $\left(\sum_{T\in\Xi} \eta_\Tau^2(V;T)\right)^\frac12$ by $\eta_\Tau(V;\Xi)$.


The next lemma establishes a local lower bound for the error. Its proof follows the usual techniques taking into account that if $u$ denotes the solution of problem~\eqref{E:cont prob},
\begin{equation*}
|\langle\RRR(V),v\rangle|=|a(V;V,v)-L(v)|=|a(V;V,v)-a(u;u,v)|
\le C_A\|\nabla(V-u)\|_{\omega}\|\nabla v\|_{\omega},
\end{equation*}
for $V\in\VV_\Tau$, whenever $v\in H^1_0(\Omega)$ vanishes outside of $\omega$, for any $\omega\subset\overline{\Omega}$.
\begin{lemma}[Local lower bound]\label{L:local lower bound}
Let $u\in H^1_0(\Omega)$ be the solution of problem~\eqref{E:cont prob}. Let $\Tau\in\TT$ and $T\in\Tau$ be fixed. If $V\in\VV_\Tau$,\footnote{From now on, we will write $a \lesssim b$ to indicate that $a \le C b$ with $C>0$ a constant depending on the data of the problem and possibly on shape regularity $\kappa_\TT$ of the meshes. Also $a \simeq b$ will indicate that $a \lesssim b$ and $b \lesssim a$.}
\begin{equation}\label{E:local lower bound}
 \est{V}{T}\lesssim  \|\nabla(V-u)\|_{\omega_\Tau(T)} +\h_T\normNT{R_\Tau(V)-\overline{R_\Tau(V)}}+\h_T^{\frac12}\normbT{J_\Tau(V)-\overline{J_\Tau(V)}},
\end{equation}
where $\overline{R_\Tau(V)}_{|_{T'}}$ denotes the mean value of $R_\Tau(V)$ on $T'$, for all $T'\in\mathcal{N}_\Tau(T)$, and for each side $S\subset\partial T$, $\overline{J_\Tau(V)}_{|_{S}}$ denotes the mean value of $J_\Tau(V)$ on $S$.
\end{lemma}

The last result is known as \emph{local efficiency of the error estimator}. According to the lemma, if a local estimator is large, then so is the corresponding local error, provided the last two terms in the right-hand side of~\eqref{E:local lower bound} are relatively small.

We define the \emph{local oscillation} corresponding to $V\in\VV_\Tau$ by
\begin{equation*}
\osci_\Tau^2(V;T)\definedas \h_T^2\normT{R_\Tau(V)-\overline{R_\Tau(V)}}^2+\h_T\normbT{J_\Tau(V)-\overline{J_\Tau(V)}}^2, \quad\forall\,T\in\Tau,
\end{equation*}
and the \emph{global oscillation} by
$$
\osci_\Tau^2(V)\definedas \sum_{T\in\Tau} \osci_\Tau^2(V;T).
$$
In general, if $\Xi\subset\Tau$ we denote $\left(\sum_{T\in\Xi} \osci_\Tau^2(V;T)\right)^\frac12$ by $\osci_\Tau(V;\Xi)$.

As an immediate consequence of the last lemma, adding over all elements in the mesh we obtain the following

\begin{theorem}[Global lower bound]\label{T:cota inferior global}
Let $u\in H^1_0(\Omega)$ denote the solution of problem~\eqref{E:cont prob}. Then, there exists a constant $C_L=C_L(d,\kappa_\TT,C_A)>0$ such that 
$$C_L\eta_\Tau^2(V)\le \|\nabla (V-u)\|_\Omega^2+\osci_\Tau^2(V),\qquad\forall\, V\in\VV_\Tau,\quad\forall\,\Tau\in\TT.$$
\end{theorem}

We conclude this section with two estimations for the error, whose proofs are strongly based on the analogous results for linear elliptic problems (cf.~\cite{CKNS-quasi-opt}).

\begin{theorem}[Global upper bound]\label{T:global upper bound}
Let $u\in H^1_0(\Omega)$ be the solution of problem~\eqref{E:cont prob}. Let $\Tau\in\TT$ and let $U\in\VV_\Tau$ be the solution of the discrete problem~\eqref{E:disc prob}. Then, there exists $C_U=C_U(d,\kappa_\TT,C_a,c_a,c_A)>0$ such that
\begin{equation}\label{E:global upper bound}
\|\nabla (U-u)\|_\Omega^2\le C_U \eta_\Tau^2(U).
\end{equation}
\end{theorem}

\begin{proof}
Let $u\in H^1_0(\Omega)$ be the solution of problem~\eqref{E:cont prob}. Let $\Tau\in\TT$ and let  $U\in\VV_\Tau$ be the solution of the discrete problem~\eqref{E:disc prob}. Let $w\in H^1_0(\Omega)$ be the solution of the \emph{linear} elliptic problem
\begin{equation}\label{E:cota superior aux}
a(U;w,v)=L(v),\qquad\forall\,v\in H^1_0(\Omega).
\end{equation}
Since $A$ is strongly monotone (cf.~\eqref{E:A strongly mon}), using that $u$ is the solution of problem~\eqref{E:cont prob},~\eqref{E:cota superior aux}, and that $a$ is bounded (cf.~\eqref{E:a bounded}), we have that
\begin{align*}
c_A\|\nabla (U-u)\|_\Omega^2&\le \langle AU -Au, U-u\rangle= a(U;U,U-u)-a(u;u,U-u)\\
&=a(U;U,U-u)-a(U;w,U-u)=a(U;U-w,U-u)\\
&\le C_a\|\nabla(U-w)\|_\Omega\|\nabla (U-u)\|_\Omega,
\end{align*}
and thus,
$$\|\nabla (U-u)\|_\Omega\le\frac{C_a}{c_A}\|\nabla(U-w)\|_\Omega.$$
Since $U$ is solution of the Galerkin discretization of the linear elliptic problem~\eqref{E:cota superior aux} in $\VV_\Tau$ (see~\eqref{E:disc prob}), using the \emph{reliability} of the global error estimator for linear problems (cf.~\cite[Lemma 2.2]{CKNS-quasi-opt}), it follows that there exists $C_U=C_U(d,\kappa_\TT,C_a,c_a,c_A)>0$ such that~\eqref{E:global upper bound} holds.
\end{proof}

\begin{theorem}[Localized upper bound]\label{T:localized upper bound}
Let $\Tau\in\TT$ and let $\Taustar\in\TT$ be a refinement of $\Tau$. Let $\mathcal{R}$ denote the subset of $\Tau$ consisting of the elements which are refined to obtain $\Taustar$, that is, $\mathcal{R}\definedas\{T\in\Tau\mid T\not\in\Taustar\}$. Let $U\in\VV_\Tau$ and $U_*\in\VV_\Taustar$ be the solutions of the discrete problem~\eqref{E:disc prob} in $\VV_\Tau$ and $\VV_\Taustar$, respectively. Then, there exists a constant $C_{LU}=C_{LU}(d,\kappa_\TT,C_a,c_a,c_A)>0$ such that
\begin{equation}\label{E:localized upper bound}
\|\nabla(U-U_*)\|_\Omega^2\le C_{LU} \eta_\Tau^2(U;\mathcal{R}).
\end{equation}
\end{theorem}

\begin{proof}
Let $\Tau$, $\Taustar$, $\mathcal{R}$, $U$ and $U_*$ be as in the assumptions of the theorem. Let $W_*\in \VV_\Taustar$ be the solution of the discrete linear elliptic problem
\begin{equation}\label{E:cota superior local aux}
a(U;W_*,V_*)=L(V_*),\qquad\forall\,V_*\in\VV_\Taustar.
\end{equation}
Analogously to the last proof, using that $A$ is strongly monotone, that $U_*$ is the solution of problem~\eqref{E:disc prob} in $\VV_\Taustar$,~\eqref{E:cota superior local aux} and that $a$ is bounded, we have that
\begin{align*}
c_A\|\nabla(U-U_*)\|_\Omega^2&\le \langle AU -AU_*, U-U_*\rangle= a(U;U,U-U_*)-a(U_*;U_*,U-U_*)\\
&=a(U;U,U-U_*)-a(U;W_*,U-U_*)=a(U;U-W_*,U-U_*)\\
&\le C_a\|\nabla(U-W_*)\|_\Omega\|\nabla(U-U_*)\|_\Omega,
\end{align*}
and therefore,
$$\|\nabla(U-U_*)\|_\Omega\le\frac{C_a}{c_A}\|\nabla(U-W_*)\|_\Omega.$$
Finally, since $U$ and $W_*$ are the solutions of the Galerkin discretization of the linear elliptic problem~\eqref{E:cota superior aux} in $\VV_\Tau$ and $\VV_\Taustar$, respectively (cf.~\eqref{E:disc prob} and~\eqref{E:cota superior local aux}), using the \emph{localized upper bound} for linear problems (cf.~\cite[Lemma 3.6]{CKNS-quasi-opt}), it follows that there exists $C_{LU}=C_{LU}(d,\kappa_\TT,C_a,c_a,c_A)>0$ such that~\eqref{E:localized upper bound} holds.
\end{proof}


\subsection{Estimator reduction and perturbation of oscillation}\label{S:reducciones nolineal}

In order to prove the contraction property it is necessary to study the effects that refinement has upon the error estimators and oscillation terms. We thus present two main results in this section. The first one is related to the error estimator and it will be used in Theorem~\ref{T:propiedad de contraccion}.

\begin{proposition}[Estimator reduction]\label{P:reduccion del estimador}
Let $\Tau\in\TT$ and let $\MM$ be any subset of $\Tau$. Let $\Taustar\in\TT$ be obtained from $\Tau$ by bisecting at least $\nref\ge 1$ times each element in $\MM$. 
If $V\in\VV_\Tau$ and $V_*\in\VV_\Taustar$, then
\begin{equation*}
\eta_\Taustar^2(V_*) \leq (1+\delta )\left\{ \eta_\Tau^2(V)-(1-2^{-\frac{\nref}{d}})\eta_\Tau^2(V;\MM)\right\}+(1+\delta ^{-1})C_E\|\nabla(V_*-V)\|_\Omega^2,
\end{equation*}
for all $\delta >0$, where $C_E>1$ is a constant (cf. Lemma~\ref{L:bound for gtau} below).
\end{proposition}

The second result is related to the oscillation terms. It will be used to establish the \emph{quasi-optimality for the error} (see Lemma~\ref{L:generalized Cea}) and to prove Lemma~\ref{L:optimal marking} in the next section. 

\begin{proposition}[Oscillation perturbation]\label{P:perturbacion de la oscilacion}
Let $\Tau\in\TT$ and let $\Taustar\in\TT$ be a refinement of $\Tau$. If $V\in\VV_\Tau$ and $V_*\in\VV_\Taustar$, then
\begin{equation*}
\osci_\Tau^2(V;\Tau\cap\Taustar) \leq 2\osci_\Taustar^2(V_*;\Tau\cap\Taustar)+2C_E\|\nabla(V_*-V)\|_\Omega^2,
\end{equation*}
where $C_E>1$ is a constant (cf. Lemma~\ref{L:bound for gtau} below).
\end{proposition}

In order to prove Propositions~\ref{P:reduccion del estimador} and~\ref{P:perturbacion de la oscilacion} we observe that if we define for $\Tau\in\TT$ and $V,W\in\VV_\Tau$
\begin{equation}\label{E:gtau}
g_\Tau(V,W;T)\definedas \h_T\normT{R_\Tau(V)-R_\Tau(W)} +
\h_T^{\frac12}\normbT{J_\Tau(V)-J_\Tau(W)},
\end{equation}
then from the definition of the local error estimators~\eqref{E:local estimator} and the triangle inequality it follows that
\begin{equation}\label{E:est-comparison}
\est{W}{T}\leq  \est{V}{T}+g_\Tau(V,W;T),\qquad\forall\,T\in\Tau,
\end{equation}
and analogously
\begin{equation}\label{E:osc-comparison}
\osc{W}{T}\leq  \osc{V}{T}+g_\Tau(V,W;T),\qquad\forall\,T\in\Tau.
\end{equation}
After proving that $g_\Tau(V,W;T)$ is bounded by $\|\nabla(V-W)\|_{\omega_\Tau(T)}$, the first terms on the right-hand sides of~\eqref{E:est-comparison} and~\eqref{E:osc-comparison} may be treated as in~\cite[Corollary 3.4 and Corollary 3.5]{CKNS-quasi-opt} for linear elliptic problems, respectively, and the assertions of Propositions~\ref{P:reduccion del estimador} and~\ref{P:perturbacion de la oscilacion} follow. On the other hand, while proving that $g_\Tau(V,W;T) \lesssim \|\nabla(V-W)\|_{\omega_\Tau(T)}$ is easy for linear problems by using inverse inequalities and trace theorems, it is not so obvious for nonlinear problems. Therefore, we omit the details of the proofs of the last two propositions, but we prove the following lemma, which is the main difference with linear problems~\cite{CKNS-quasi-opt}.

\begin{lemma}\label{L:bound for gtau}
Let $\Tau\in\TT$ and let $g_\Tau$ be given by~\eqref{E:gtau}. Then, there holds that
\begin{equation}\label{E:local bound for gtau}
g_\Tau(V,W;T)\lesssim \|\nabla(V-W)\|_{\omega_\Tau(T)},\qquad\forall\,V,W\in\VV_\Tau,\quad\forall\,T\in\Tau.
\end{equation}
Consequently, there exists a constant $C_E>1$ which depends on $d$, $\kappa_\TT$ and the problem data, such that
\begin{equation}\label{E:bound for gtau}
\sum_{T\in\Tau} g_\Tau^2(V,W;T)\le C_E \|\nabla(V-W)\|_\Omega^2,\qquad\forall\,V,W\in\VV_\Tau.
\end{equation}
\end{lemma}

In order to prove Lemma~\ref{L:bound for gtau}, we define
\begin{equation}\label{E:Gamma de v}
\Gamma_{V}(x)\definedas\nabla_2\gamma(x,\nabla V(x))= \alpha(x,|\nabla V(x)|^2)\nabla V(x),\qquad\forall\,x\in\Omega,
\end{equation}
and prove first the following auxiliary result.

\begin{lemma}\label{L:auxiliar acotacion de gtau}
Let $T\in\Tau$. Let $D_2^2\gamma$ be the Hessian matrix of $\gamma$ as a function of its second variable. If
$$\|D_2^2\gamma(x,\xi)-D_2^2\gamma(y,\xi)\|_2\le C_\gamma |x-y|,\qquad\forall\,x,y\in T,\,\xi\in\RR^d,$$
for some constant $C_\gamma>0$, then for all $V,W\in\PP_1(T)$, there holds that
$$|\Gamma_{V}(x)-\Gamma_{W}(x)-\Gamma_V(y)+\Gamma_W(y)|\le C_\gamma\|\nabla(V-W)\|_{L^\infty(T)}|x-y|,\qquad\forall\,x,y\in T.$$
\end{lemma}

\begin{remark}
Taking into account~\eqref{E:gradiente de gamma}, we have that
$$(D_2^2\gamma(x,\xi))_{ij}=2D_2\alpha(x,|\xi|^2) \xi_i\xi_j+\alpha(x,|\xi|^2) \delta_{ij},$$
for $1\le i,j\le d$, where $\delta_{ij}$ denotes the Kronecker's delta. In consequence, if $\alpha(\cdot,t)$ and $D_2 \alpha(\cdot,t)t$ are Lipschitz on each $T\in\Tau_0$ uniformly in $t>0$, it follows that $D_2^2\gamma(x,\xi)$ is locally Lipschitz as a function of its first variable, i.e., there exists a constant $C_\gamma>0$ such that
$$\|D_2^2\gamma(x,\xi)-D_2^2\gamma(y,\xi)\|_2\le C_\gamma |x-y|,\qquad\forall\,x,y\in T,\,\xi\in\RR^d,$$
for all $T\in\Tau_0$. In particular this holds for any $T\in\Tau$, $\Tau \in \TT$.
\end{remark}

\begin{proof}[Proof of Lemma~\ref{L:auxiliar acotacion de gtau}]
Let $T\in\Tau$. Let $V,W\in\PP_1(T)$ and $x,y\in T$. Taking into account that $V$ and $W$ are linear over $T$, we denote $\mathbf{v}\definedas \nabla V(x)=\nabla V(y)$ and $\mathbf{w}\definedas \nabla W(x)=\nabla W(y)$. Thus, we have that
\begin{align*}
|\Gamma_{V}(x)&-\Gamma_{W}(x)-\Gamma_V(y)+\Gamma_W(y)|=|\nabla_2\gamma(x,\mathbf{v})-\nabla_2\gamma(x,\mathbf{w})-\nabla_2\gamma(y,\mathbf{v})+\nabla_2\gamma(y,\mathbf{w})|\\
&=\left|\int_0^1 \left[ D^2_2 \gamma(x,\mathbf{w}+r(\mathbf{v}-\mathbf{w}))-D^2_2 \gamma(y,\mathbf{w}+r(\mathbf{v}-\mathbf{w}))\right](\mathbf{v}-\mathbf{w})~dr\right|\\
&\le C_\gamma |x-y||\mathbf{v}-\mathbf{w}|,
\end{align*}
which completes the proof of the lemma.
\end{proof}

We conclude this section with the proof of Lemma~\ref{L:bound for gtau}, where we use that
\begin{equation*}
{R_\Tau(V)}_{|_{T}}= -\nabla \cdot \Gamma_V-f,\qquad\text{and}\qquad
{J_\Tau(V)}_{|_{S}}=\frac12\left({\Gamma_{V}}_{|_{T_1}}\cdot \vec{n}_1+{\Gamma_{V}}_{|_{T_2}}\cdot \vec{n}_2\right),
\quad S \subset \Omega,
\end{equation*}
which is an immediate consequence of~\eqref{E:Gamma de v} and the definitions of the element residual~\eqref{E:element-residual} and the jump residual~\eqref{E:jump-residual}. 

\begin{proof}[Proof of Lemma~\ref{L:bound for gtau}]
Let $\Tau\in\TT$ and let $V,W\in\VV_\Tau$. Let $T\in\Tau$ be fixed.

\step{1} By Lemma~\ref{L:auxiliar acotacion de gtau}, for the element residual we have that
\begin{align*}
\normT{R_\Tau(V)-R_\Tau(W)}&= \normT{\nabla \cdot (\Gamma_{V}-\Gamma_W)}\le \h_T^{\frac{d}{2}} \|\nabla \cdot (\Gamma_{V}-\Gamma_{W})\|_{L^\infty(T)}\\
&\lesssim\h_T^{\frac{d}{2}}\sup_{\substack{x,y\in T\\x\neq y}}\frac{|\Gamma_{V}(x)-\Gamma_{W}(x)-\Gamma_V(y)+\Gamma_W(y)|}{|x-y|}\\
&\lesssim\h_T^{\frac{d}{2}}\|\nabla(V-W)\|_{L^\infty(T)}=\|\nabla(V-W)\|_T,
\end{align*}
and thus,
\begin{equation}\label{E:gtau aux1}
\h_T\normT{R_\Tau(V)-R_\Tau(W)}\lesssim \|\nabla (V-W)\|_T.
\end{equation}

\step{2} Consider now the term corresponding to the jump residual. If $S$ is a side of $T$ which is interior to $\Omega$ and if $T_1$ and $T_2$ are the elements sharing $S$, we have that
\begin{align*}
\normS{J_\Tau(V)-J_\Tau(W)}&=\normS{\frac12\sum_{i=1,2}{(\Gamma_{V}-\Gamma_W)}_{|_{T_i}}\cdot \vec{n}_i}\leq \sum_{i=1,2}\normS{{(\Gamma_{V}-\Gamma_W)}_{|_{T_i}} }\\
&\lesssim \sum_{i=1,2}\left(\h_T^{-\frac12}\|\Gamma_{V}-\Gamma_{W}\|_{T_i}+\h_T^{\frac12}\|\nabla (\Gamma_{V}-\Gamma_{W})\|_{T_i}\right),
\end{align*}
where we have used a scaled trace theorem. Since $\nabla_2 \gamma$ is Lipschitz as a function of its second variable, we have that
$$|\Gamma_{V}(x)-\Gamma_W(x)|=|\nabla_2\gamma(x,\nabla V(x))-\nabla_2\gamma(x,\nabla W(x))|\lesssim|\nabla V(x)-\nabla W(x)|,$$
for $x\in T_i$ ($i=1,2$), and therefore,
$$\|\Gamma_{V}-\Gamma_W\|_{T_i}\lesssim\|\nabla (V-W)\|_{T_i},\qquad i=1,2.$$ 
Using the same argument as in~\step{1}, we have that $\|\nabla  (\Gamma_{V}-\Gamma_{W})\|_{T_i}\lesssim \|\nabla (V-W)\|_{T_i}$, for $i=1,2$, and in consequence,
\begin{equation}\label{E:gtau aux2}
\h_T^{\frac12}\normbT{J_\Tau(V)-J_\Tau(W)}\lesssim \|\nabla (V-W)\|_{\omega_\Tau(T)}.
\end{equation}
Finally,~\eqref{E:local bound for gtau} follows from~\eqref{E:gtau aux1} and~\eqref{E:gtau aux2}, taking into account~\eqref{E:gtau}.
\end{proof}

\section{Linear convergence of an adaptive FEM}\label{S:convergence}

In this section we present the adaptive FEM and establish one of the main results of this article (Theorem~\ref{T:propiedad de contraccion} below) which guarantees the convergence of the adaptive sequence.

\subsection{The adaptive loop}

We consider the following adaptive loop to approximate the solution $u$ of problem~\eqref{E:cont prob}.

\medskip
\begin{center}
\doublebox{ 
\begin{minipage}{.9\textwidth}
\textbf{Adaptive Algorithm.}
\tt
Let $\Tau_0$ be an initial conforming mesh of $\Omega$ and let $\theta$ be a parameter satisfying $0<\theta<1$. Let $k=0$.
\begin{enumerate}
\vspace{0.15cm}
 \item [1.] $U_k\definedas \textsf{SOLVE}(\Tau_k)$.
\vspace{0.15cm}
\item [2.] $\{\eta_k(T)\}_{T\in\Tau_k}\definedas \textsf{ESTIMATE}(U_k,\Tau_k)$.
\vspace{0.15cm}
\item [3.] $\MM_k\definedas \textsf{MARK}(\{\eta_k(T)\}_{T\in\Tau_k},\Tau_k,\theta)$.
\vspace{0.15cm}
\item [4.] $\Tau_{k+1}\definedas \textsf{REFINE}(\Tau_k,\MM_k,n)$.
\vspace{0.15cm}
\item [5.] Increment $k$ and go back to step 1.
\end{enumerate}
\end{minipage}
}
\end{center}

Now we explain each module in the last algorithm.

\begin{itemize}
\item \textbf{The module \textsf{SOLVE}.} This module takes a conforming triangulation $\Tau_k$ of $\Omega$ as input argument and outputs the solution $U_k$ of the discrete problem~\eqref{E:disc prob} in $\Tau_k$; i.e., $U_k\in\VV_k\definedas \VV_{\Tau_k}$ satisfies
\begin{equation*}
a(U_k;U_k,V)=L(V),\qquad \forall~V\in  \VV_k.
\end{equation*}

\item \textbf{The module \textsf{ESTIMATE}.} This module computes the a posteriori local error estimators $\eta_k(T)$ of $U_k$ over $\Tau_k$ given by
$\eta_k(T)\definedas \eta_{\Tau_k}(U_k;T),$ for all $T\in\Tau_k,$
(see~\eqref{E:local estimator}).

\item \textbf{The module \textsf{MARK}.} Based on the local error estimators, the module \textsf{MARK} selects a subset $\MM_k$ of $\Tau_k$, using an \emph{efficient D\"orfler's strategy}. More precisely, given the marking parameter $\theta\in (0,1)$, the module \textsf{MARK} selects a \emph{minimal} subset $\MM_k$ of $\Tau_k$ such that
\begin{equation}\label{E:dorfler}
 \eta_k(\MM_k)\ge \theta \, \eta_k(\Tau_k),
\end{equation}
where $\eta_k(\MM_k)=\left(\sum_{T\in\MM_k}\eta_k^2(T)\right)^{\frac12}$ and $\eta_k(\Tau_k)=\left(\sum_{T\in\Tau_k}\eta_k^2(T)\right)^{\frac12}$.

\item \textbf{The module \textsf{REFINE}.} Finally, the module \textsf{REFINE} takes the mesh $\Tau_k$ and the subset $\MM_k\subset \Tau_k$ as inputs. By using the bisection rule described by Stevenson in~\cite{Stevenson-refine}, this module refines (bisects) $n$ times (where $n \ge 1$ is fixed) each element in $\MM_k$. After that, with the goal of keeping conformity of the mesh, possibly some further bisections are performed leading to a new conforming triangulation $\Tau_{k+1} \in \TT$ of $\Omega$, which is a refinement of $\Tau_k$ and the output of this module.
\end{itemize}

From now on, $U_k$, $\{\eta_k(T)\}_{T\in\Tau_k}$, $\MM_k$, $\Tau_k$ will denote the outputs of the corresponding modules \textsf{SOLVE}, \textsf{ESTIMATE}, \textsf{MARK} and \textsf{REFINE}, when iterated after starting with a given initial mesh $\Tau_0$.

\subsection{An equivalent notion for the error}

In order to prove a contraction property for the error of a similar AFEM for linear elliptic problems the well-known \emph{Galerkin orthogonality relationship} is used(see~\cite{CKNS-quasi-opt}). In this case, due to the nonlinearity of our problem, this property does not hold. We present an equivalent notion of error so that it is possible to establish a property analogous to the orthogonality (cf.~\eqref{E:orthogonality} below).

It is easy to check that $\JJ:H^1_0(\Omega)\to\RR$ given by
\begin{equation*}
\JJ(v)\definedas \int_0^1 \langle A(rv),v\rangle~dr=\int_\Omega \gamma(\cdot,\nabla v)~dx,\quad\forall\,v\in H^1_0(\Omega),
\end{equation*}
is a potential for the operator $A$. More precisely, if $\WW$ is a closed subspace of $H^1_0(\Omega)$, the following claims are equivalent
\begin{itemize}
 \item $w\in\WW$ is solution of
\begin{equation}\label{E:problema en w}
a(w;w,v)=L(v),\qquad\forall\,v\in\WW,
\end{equation}
where $L(v)=\int_\Omega fv$, for $v\in H^1_0(\Omega)$.
\item $w\in\WW$ minimizes the functional $\FF:H^1_0(\Omega)\to\RR$ over $\WW$, where $\FF$ is given by
\begin{equation}\label{E:funcional F beta}
\FF(v)\definedas\JJ(v)-L(v)=\int_\Omega \gamma(\cdot,\nabla v)-fv \, dx,\quad\, v\in H^1_0(\Omega).
\end{equation}
\end{itemize}

The following theorem states a notion equivalent to the $H^1_0(\Omega)$-error. The proof follows the ideas used in~\cite{diening-kreuzer} and uses that the Hessian matrix of $\gamma$, denoted by $D^2_2\gamma$, is uniformly elliptic, i.e.,
\begin{equation}\label{E:hessiano de gamma uniformemente eliptico}
c_A|\zeta|^2\le D^2_2\gamma(x,\xi)\zeta\cdot\zeta\le  C_A|\zeta|^2,\qquad\forall\,x\in\Omega,\,\xi,\zeta\in\RR^d.
\end{equation}
This fact holds because $\nabla_2 \gamma$ is Lipschitz and strongly monotone as a function of its second variable.

\begin{theorem}\label{T:equivalence for error}
Let $\WW$ be a closed subspace of $H^1_0(\Omega)$ and let $\FF$ be given by~\eqref{E:funcional F beta}. If $w\in\WW$ satisfies~\eqref{E:problema en w}, then 
$$\frac{c_A}{2}\|\nabla(v-w)\|_\Omega^2\le\FF(v)-\FF(w)\le \frac{C_A}{2}\|\nabla(v-w)\|_\Omega^2,\quad\,\forall\,v\in\WW.$$

%
\end{theorem}

\begin{proof}
Let $\WW$ be a closed subspace of $H^1_0(\Omega)$ and let $w\in\WW$ be the solution of~\eqref{E:problema en w}. Let $v\in\WW$ be fixed and arbitrary. For $z\in\RR$, we define $\phi(z)\definedas (1-z)w+zv$, and note that
\begin{equation*}
\phi'(z)=v-w\qquad\text{and}\qquad\nabla \phi(z)=(1-z)\nabla w+z\nabla v.
\end{equation*}
If we define $\psi(z)\definedas \FF(\phi(z))$, integration by parts yields
\begin{equation}\label{E:taylor}
\FF(v)-\FF(w)=\psi(1)-\psi(0)=\psi'(0)+ \int_0^1 \psi''(z)(1-z)~dz.
\end{equation}
From~\eqref{E:funcional F beta} it follows that
\begin{equation}\label{E:psidet}
\psi(z)=\FF(\phi(z))=\int_\Omega \gamma(x,\nabla \phi(z))~dx-\int_\Omega f\phi(z)~dx, 
\end{equation}
and therefore, in order to obtain the derivatives of $\psi$ we first compute $\frac{\partial}{\partial z} (\gamma(x,\nabla \phi(z)))$, for each $x\in\Omega$ fixed.
On the one hand, we have that
\begin{align*}
\frac{\partial}{\partial z} \gamma(\cdot,\nabla \phi(z))&=\nabla_2\gamma(\cdot,\nabla \phi(z))\cdot \frac{\partial}{\partial z}\nabla \phi(z)=\nabla_2\gamma(\cdot,\nabla \phi(z))\cdot \nabla (v-w),
\end{align*}
and then
\begin{align*}
\frac{\partial^2}{\partial z^2} \gamma(\cdot,\nabla \phi(z))
&=D^2_2 \gamma(\cdot,\nabla \phi(z))\nabla(v-w)\cdot\nabla(v-w),
\end{align*}
where $D^2_2 \gamma
$ is the Hessian matrix of $\gamma$ as a function of its second variable. Thus, taking into account that $\phi''(z)=0$ for all $z\in\RR$, from~\eqref{E:psidet} it follows that
\begin{equation}\label{E:psi segunda}
\psi''(z)=\int_\Omega D^2_2 \gamma(x,\nabla \phi(z))\nabla(v-w)\cdot\nabla(v-w)~dx.
\end{equation}
Since $w$ minimizes $\FF$ over $\WW$, we have that $\psi'(0)=0$; and using~\eqref{E:psi segunda}, from~\eqref{E:taylor} we obtain that
\begin{align*}
\FF(v)-\FF(w)&= \int_0^1\int_\Omega D^2_2 \gamma(x,\nabla \phi(z))\nabla(v-w)\cdot\nabla(v-w) (1-z)~dx~dz.
\end{align*}
Finally, since $D^2_2\gamma$ is uniformly elliptic (cf.~\eqref{E:hessiano de gamma uniformemente eliptico}) we have that
$$\frac{c_A}{2}\|\nabla(v-w)\|_\Omega^2\le\int_0^1\int_\Omega D^2_2 \gamma(x,\nabla \phi(z))\nabla(v-w)\cdot\nabla(v-w) (1-z)~dx~dz\le\frac{C_A}{2}\|\nabla(v-w)\|_\Omega^2,$$
which concludes the proof.
\end{proof}

As an immediate consequence of the last theorem,
\begin{equation}\label{E:equivalence for error}
\frac{c_A}{2}\|\nabla(U_k-U_p)\|_\Omega^2\le\FF(U_k)-\FF(U_p)\le \frac{C_A}{2} \|\nabla(U_k-U_p)\|_\Omega^2,\qquad\forall\,k,p\in\NN_0,\,k<p,
\end{equation}
and the same estimation holds replacing $U_p$ by $u$, the exact weak solution of problem~\eqref{E:cont prob}.

\subsection{Convergence of the adaptive FEM}

Recall that $u$ denotes the exact weak solution of problem~\eqref{E:cont prob}, and $U_k$, $\{\eta_k(T)\}_{T\in\Tau_k}$, $\MM_k$, $\Tau_k$ will denote the outputs of the corresponding modules \textsf{SOLVE}, \textsf{ESTIMATE}, \textsf{MARK} and \textsf{REFINE} of the Adaptive Algorithm when iterated after starting with a given initial mesh $\Tau_0$.

Taking into account the estimator reduction (Proposition~\ref{P:reduccion del estimador}), the global upper bound (Theorem~\ref{T:global upper bound}) and~\eqref{E:equivalence for error}, we now prove the following result which establish the convergence of the Adaptive Algorithm.

\begin{theorem}[Contraction property]\label{T:propiedad de contraccion}
There exist constants $0<\rho<1$ and $\mu>0$ which depend on $d$, $\kappa_\TT$, of problem data, of number of refinements $n$ performed on each marked element and the marking parameter $\theta$ such that
$$[\FF(U_{k+1})-\FF(u)]+\mu \eta_{k+1}^{2} \leq \rho^2([\FF(U_k)-\FF(u)]+\mu \eta _k^{2}),\qquad\forall\,k\in\NN_0,$$
where $\eta_k\definedas\left(\sum_{T\in\Tau_k}\eta_k^2(T)\right)^{\frac12}$ denotes the global error estimator in $\Tau_k$.
\end{theorem}

\proof
Let $k \in \NN_0$, using that
\begin{equation}\label{E:orthogonality}
\FF(U_k)-\FF(u)=\FF(U_k)-\FF(U_{k+1})+\FF(U_{k+1})-\FF(u), 
\end{equation}
and the estimator reduction given by Proposition~\ref{P:reduccion del estimador} with $\Tau=\Tau_k$ and $\Taustar=\Tau_{k+1}$ we have that
\begin{align*}
[\FF(U_{k+1})-\FF(u)]+\mu \eta_{k+1}^{2} &\leq [\FF(U_k)-\FF(u)]-[\FF(U_k)-\FF(U_{k+1})]\\
&\qquad+(1+\delta
)\mu \left\{ \eta _k^{2}-\xi\eta _k^2(\MM_k)\right\} +(1+\delta ^{-1})C_E\mu \|\nabla (U_k-U_{k+1})\|_\Omega^{2},
\end{align*}
for all $\delta,\mu>0$, where $\xi\definedas 1-2^{-\frac{\nref}{d}}$ and $\eta_k^2(\MM_k)\definedas \sum_{T\in\MM_k}\eta_k^2(T)$. By choosing $\mu\definedas \frac{c_A}{2(1+\delta ^{-1})C_E}$, and using~\eqref{E:equivalence for error} it follows that
\begin{equation*}
[\FF(U_{k+1})-\FF(u)]+\mu \eta_{k+1}^{2} \leq [\FF(U_k)-\FF(u)]+(1+\delta
)\mu \left\{ \eta _k^{2}-\xi\eta _k^2(\MM_k)\right\}.
\end{equation*}
D\"orfler's strategy yields $\eta_k(\MM_k)\geq \theta\eta_k$ and thus
\begin{align*}
[\FF(U_{k+1})-\FF(u)]+\mu \eta_{k+1}^{2} &\leq [\FF(U_k)-\FF(u)]+(1+\delta
)\mu \eta _k^{2}-(1+\delta
)\mu \xi\theta^2\eta _k^{2}\\
& = [\FF(U_k)-\FF(u)]
+(1+\delta)\mu \left( 1 -\frac{\xi\theta^2}{2}\right) \eta _k^{2}-(1+\delta
)\mu \frac{\xi\theta^2}{2}\eta _k^{2}.
\end{align*}
Using~\eqref{E:equivalence for error}, the global upper bound (Theorem~\ref{T:global upper bound}) and that $ (1+\delta)\mu=\frac{c_A\delta}{2C_E}$ it follows that
\begin{equation*}
[\FF(U_{k+1})-\FF(u)]+\mu \eta_{k+1}^{2} 
\leq [\FF(U_k)-\FF(u)]
+(1+\delta )\mu\left(1- \frac{\xi\theta^2}{2}\right)\eta _k^{2}
- \frac{c_A\delta \xi\theta^2}{2C_UC_EC_A}[\FF(U_k)-\FF(u)].
\end{equation*}
If we define
$$\rho_1^2(\delta)\definedas\left( 1- \frac{c_A\delta \xi\theta^{2}}{2C_U C_EC_A}\right),\qquad \rho_2^2(\delta)\definedas\left(
1-\frac{\xi\theta^{2}}{2}\right)(1+\delta ),$$
we thus have that
$$[\FF(U_{k+1})-\FF(u)]+\mu \eta_{k+1}^{2} \leq \rho_1^2(\delta)[\FF(U_k)-\FF(u)]+\mu \rho_2^2(\delta) \eta _k^{2}.$$
The proof concludes choosing $\delta > 0$ small enough to satisfy
\begin{equation*}\tag*{\qedsymbol}
0<\rho \definedas\max \{\rho_{1}(\delta),\rho_{2}(\delta)\}<1.
\end{equation*}

The last result, coupled with~\eqref{E:equivalence for error} allows us to conclude that the sequence $\{U_k\}_{k\in\NN_0}$ of discrete solutions obtained through the Adaptive Algorithm converges to the weak solution $u$ of the nonlinear problem~\eqref{E:cont prob}, and moreover, there exists $\rho\in (0,1)$ such that
$$\|\nabla(U_k-u)\|_\Omega\le C \rho^k,\qquad\forall\,k\in\NN_0,$$
for some constant $C>0$. Also, the global error estimators $\{\eta_k\}_{k\in\NN_0}$ tend to zero, and in particular,
$$\eta_k\le C \rho^k,\qquad\forall\,k\in\NN_0,$$
for some constant $C>0$.

\section{Optimality of the total error and optimal marking}

In this section we introduce the notion of \emph{total error}, we show an analogous of Cea's lemma for this new notion (see Lemma~\ref{L:generalized Cea}) and a result about \emph{optimal marking} (see Lemma~\ref{L:optimal marking}). Both of them will be very important to establish a control of marked elements in each step of the adaptive procedure (cf. Lemma~\ref{L:cardinality} in Section~\ref{S:optimality}).

We first present an auxiliary result that will allow us to show the analogous of Cea's lemma for the \emph{total error}. Its proof is an immediate consequence of Theorem~\ref{T:equivalence for error} and will thus be omitted.

\begin{lemma}[Quasi-orthogonality property in a mesh]\label{L:quasi-ortho mesh}
If $U\in\VV_\Tau$ denotes the solution of the discrete problem~\eqref{E:disc prob} for some $\Tau\in\TT$, then
$$\|\nabla (U-u)\|_\Omega^2+\|\nabla (U-V)\|_\Omega^2\le \frac{C_A}{c_A}\|\nabla (V-u)\|_\Omega^2,\qquad\forall\,V\in\VV_\Tau,$$
where $C_A$ and $c_A$ are the constants appearing in~\eqref{E:A Lipschitz} and~\eqref{E:A strongly mon}.
\end{lemma}


Since the global oscillation term is smaller than the global error estimator, that is, $\osci_\Tau(U)\le \eta_\Tau(U)$, using the global upper bound (Theorem~\ref{T:global upper bound}), we have that
$$\|\nabla(U-u)\|_\Omega^2+\osci_\Tau^2(U)\le (C_U+1)\eta_\Tau^2(U),$$
whenever $u$ is the solution of problem~\eqref{E:cont prob} and $U\in\VV_\Tau$ is the solution of the discrete problem~\eqref{E:disc prob}. Taking into account the global lower bound (Theorem~\ref{T:cota inferior global}) we obtain that
$$\eta_\Tau(U)\approx\left(\|\nabla(U-u)\|_\Omega^2+\osci_\Tau^2(U)\right)^{\frac12}.$$
The quantity on the right-hand side is called \emph{total error}, and since adaptive methods are based on the \emph{a posteriori error estimators}, the convergence rate is characterized through properties of the \emph{total error}. 


\begin{remark}(Cea's Lemma)
Taking into account that $A$ is Lipschitz and strongly monotone, it is easy to check that
\begin{equation*}
\|\nabla (U-u)\|_\Omega\le \frac{C_A}{c_A} \inf_{V\in \VV_\Tau}\|\nabla (V-u)\|_\Omega.
\end{equation*}
This estimation is known as Cea's Lemma and shows that the approximation $U$ is optimal (up to a constant) of the solution $u$ from $\VV_\Tau$.
\end{remark}

A generalization of Cea's Lemma for the total error is given in the following

\begin{lemma}[Cea's Lemma for the total error]\label{L:generalized Cea}
If $U\in\VV_\Tau$ denotes the solution of the discrete problem~\eqref{E:disc prob} for some $\Tau\in\TT$, then
\begin{equation*}
\|\nabla(U-u)\|_\Omega^2+\osci_\Tau^2(U)\le\frac{2C_EC_A}{c_A} \inf_{V\in \VV_\Tau}(\|\nabla(V-u)\|_\Omega^2+\osci_\Tau^2(V)),
\end{equation*}
where $C_E>1$ is the constant given in~\eqref{E:bound for gtau}.
\end{lemma}

\begin{proof}
Let $\Tau\in\TT$ and let $U\in\VV_\Tau$ be the solution of the discrete problem~\eqref{E:disc prob}. If $V\in\VV_\Tau$, using Proposition~\ref{P:perturbacion de la oscilacion} with $\Taustar=\Tau$ and Lemma~\ref{L:quasi-ortho mesh} we have that
\begin{align*}
\|\nabla (U-u)\|_\Omega^2+\osci_\Tau^2(U)&\le \|\nabla (U-u)\|_\Omega^2+2\osci_\Tau^2(V)+2C_E\|\nabla (V-U)\|_\Omega^2\\
&\le 2C_E\frac{C_A}{c_A}\|\nabla (V-u)\|_\Omega^2+2\osci_\Tau^2(V)\\
&\le \frac{2C_EC_A}{c_A}\left(\|\nabla (V-u)\|_\Omega^2+\osci_\Tau^2(V)\right).
\end{align*}
Since $V\in\VV_\Tau$ is arbitrary, the claim of this lemma follows.
\end{proof}

The following result establishes a link between nonlinear approximation theory and AFEM through D\"orfler's marking strategy. Roughly speaking, it is a reciprocal to the contraction property (Theorem~\ref{T:propiedad de contraccion}). More precisely, we prove that if there exists a suitable total error reduction from $\Tau$ to a refinement $\Taustar$, then the error indicators of the refined elements from $\Tau$ must satisfy a D\"orfler's property. In other words, D\"orfler's marking and total error reduction are intimately connected. This result is known as \emph{optimal marking} and was first proved for linear elliptic problems by Stevenson~\cite{Stevenson}. The notion of total error presented above was first introduced by Casc\'on et al.~\cite{CKNS-quasi-opt} for linear problems, together with the appropriate optimal marking result, which we mimic here.

In order to prove the \emph{optimal marking} result we assume that the marking parameter $\theta$ satisfies 
\begin{equation}\label{E:restriccion theta}
 0<\theta<\theta_0 \definedas  \left[\frac{C_L}{1+2C_{LU}(1+C_E)}\right]^{1/2},
\end{equation}
where $C_L$, $C_{LU}$ are the constants appearing in the global lower bound (Theorem~\ref{T:cota inferior global}) and in the localized upper bound (Theorem~\ref{T:localized upper bound}), respectively, and $C_E$ is the constant appearing in~\eqref{E:bound for gtau}.

\begin{lemma}[Optimal marking]\label{L:optimal marking}
Let $\Tau\in\TT$ and let $\Taustar\in\TT$ be a refinement of $\Tau$. Let $\mathcal{R}$ denote the subset of $\Tau$ consisting of the elements  which were refined to obtain $\Taustar$, i.e., $\mathcal{R} = \Tau \setminus \Taustar$. Assume that the marking parameter $\theta$ satisfies $0<\theta<\theta_0$ and define $\nu\definedas \frac12\big(1-\frac{\theta^2}{\theta_0^2}\big)>0$. Let $U$ and $U_*$ be the solutions of the discrete problem~\eqref{E:disc prob} in $\VV_\Tau$ and $\VV_\Taustar$, respectively. If
\begin{equation}\label{E:reduccion del error}
\|\nabla(U_*-u)\|_\Omega^2+\osci_\Taustar^2(U_*)\leq \nu\left( \|\nabla (U-u)\|_\Omega^2+\osci_\Tau^2(U)\right),
\end{equation}
then
$$\est{U}{\mathcal{R}}\geq \theta \gest{U}.$$
\end{lemma}

\begin{proof}
Let $\Tau$, $\Taustar$, $\mathcal{R}$, $U$, $U_*$, $\theta$ and $\nu$ be as in the assumptions. Using~\eqref{E:reduccion del error} and the global lower bound (Theorem~\ref{T:cota inferior global}) we obtain that
\begin{align}
(1-2\nu )C_{L}\eta_\Tau^2(U)&\leq (1-2\nu )\left(\|\nabla (U-u)\|_\Omega^2+\osci_\Tau^2(U)\right) \notag\\
&\leq \|\nabla (U-u)\|_\Omega^2-2\|\nabla(U_*-u)\|_\Omega^2+\osci_\Tau^2(U) -2\osci_\Taustar^2(U_*).\label{E:aux opt marking}
\end{align}
Since $\|\nabla (U-u)\|_\Omega\le \|\nabla(U_*-u)\|_\Omega+\|\nabla(U_*-U)\|_\Omega$, we have that \begin{equation}\label{E:casi orto}
\|\nabla (U-u)\|_\Omega^2 - 2\|\nabla(U_*-u)\|_\Omega^2 
\le 2\|\nabla(U_*-U)\|_\Omega^2.
\end{equation}
Using Proposition~\ref{P:perturbacion de la oscilacion} and that $\osci_\Tau^2(U;T)\leq \eta_\Tau^2(U;T)$, if $T\in\mathcal{R}=\Tau\setminus\Taustar$; for the oscillation terms we obtain that
$$\osci_\Tau^2(U) -2\osci_\Taustar^2(U_*)\leq 2C_E\|\nabla(U_*-U)\|_\Omega^2+\eta_\Tau^2(U;\mathcal{R}).$$
Taking into account~\eqref{E:casi orto} and the last inequality, from~\eqref{E:aux opt marking} it follows that
$$(1-2\nu )C_{L}\eta_\Tau^2(U)\leq
2\|\nabla(U-U_*)\|_\Omega^{2}+2C_E\|\nabla(U-U_*)\|_\Omega^{2}+\eta_\Tau^2(U;\mathcal{R}),$$ and using the localized upper bound (Theorem~\ref{T:localized upper bound}) we have that
\begin{align*}
(1-2\nu )C_{L}\eta_\Tau^2(U)&\leq 2(1+C_E)
C_{LU}\eta_\Tau^2(U;\mathcal{R})+\eta_\Tau^2(U;\mathcal{R})= (1+2C_{LU}(1+C_E))\eta_\Tau^2(U;\mathcal{R}).
\end{align*}
Finally,
\[
\frac{(1-2\nu)C_L}{1+2C_{LU}(1+C_E)}\eta_\Tau^2(U)\leq \eta_\Tau^2(U;\mathcal{R}),
\]
which completes the proof since $\frac{(1-2\nu)C_L}{1+2C_{LU}(1+C_E)}=(1-2\nu)\theta_0^2=\theta^2$ by the definition of $\nu$.
\end{proof}

\section{Quasi-optimality of the adaptive FEM}\label{S:optimality}

In this section we state the second main result of this article, that is, the adaptive sequence computed through the Adaptive Algorithm converges with optimal rate to the weak solution of the nonlinear problem~\eqref{E:cont prob}.
For $N\in\NN_0$, let $\TT_N$ be the set of all possible
conforming triangulations generated by refinement from $\Tau_0$  with at most
$N$ elements more than $\Tau_{0}$, i.e.,
\[
\TT_N\definedas\{\Tau\in \TT\mid\quad\#\Tau-\#\Tau_0\le N\}.
\]
The quality of the best approximation in $\mathbb{T}_N$ is given by
\[
\sigma_N(u)\definedas\inf_{\Tau\in \mathbb{T}_N}\inf_{V\in \VV_\Tau}\left[\|\nabla (V-u)\|_\Omega^2+\osci_\Tau^2(V)\right]^{\frac12}.
\]
For $s>0$, we say that $u\in \mathbb{A}_s$ if  
\begin{equation}\label{E:As}
|u|_{s}\definedas\sup_{N\in\NN_0}\left\{ (N+1)^{s}\sigma_N (u)\right\} <\infty.
\end{equation}
In other words, $u$ belongs to the class $\mathbb{A}_s$ if can be \emph{ideally} approximated with adaptive meshes at a rate $(DOFs)^{-s}$. From another perspective, if $u \in \mathbb{A}_s$, then for each $\varepsilon > 0$ there exist a mesh $\Tau_\varepsilon \in \TT$ and a function $V_\varepsilon \in \VV_{\Tau_\varepsilon}$ such that 
\[
\#\Tau_{\varepsilon }-\#\Tau_{0}\le |u|_{s}^{\frac{1}{s}}\varepsilon^{-\frac{1}{s}}\quad\text{ and }\quad
\|\nabla(V_\varepsilon-u)\|_\Omega^2+\osci_{\Tau_\varepsilon}^2(V_\varepsilon)\le \varepsilon^2.
\]
The study of classes of functions that will yield such rates is beyond the scope of this article. Some results along this direction can be found in~\cite{BDDP,gaspoz-morin,gaspoz-morin-ho}.

The following result proved in~\cite{Stevenson,CKNS-quasi-opt}, provides a bound for the complexity of the overlay of two triangulations $\Tau^1$ and $\Tau^2$ obtained as refinements of $\Tau_{0}$.

\begin{lemma}[Overlay of triangulations]\label{L:overlay}
For $\Tau^{1},\Tau^{2}\in \TT$ the overlay $\Tau:=\Tau^{1}\oplus \Tau^{2}\in\TT$, defined as the smallest admissible triangulation which is a refinement of $\Tau^1$ and $\Tau^2$, satisfies
\[
\#\Tau\le\#\Tau^1+\#\Tau^2-\#\Tau_{0}.
\]
\end{lemma}

The next lemma is essential for proving the main result below (see Theorem~\ref{T:opt_main}).

\begin{lemma}[Cardinality of $\MM_k$]\label{L:cardinality}
Let us assume that the weak solution $u$ of problem~\eqref{E:cont prob} belongs to $\mathbb{A}_{s}$. If the marking parameter $\theta$ satisfies $0<\theta<\theta_0$ (cf.~\eqref{E:restriccion theta}), then
$$
\#\MM_k\le \left(\frac{2C_EC_A}{\nu c_A}\right)^{\frac{1}{2s}}|u|_{s}^{\frac{1}{s}}\left[\|\nabla(U_k-u)\|_\Omega^2+\osci_{\Tau_k}^2(U_k)\right]^{-\frac{1}{2s}},\qquad\forall\,k\in\NN_0,
$$
where $\nu=\frac12\left(1-\frac{\theta^2}{\theta_0^2}\right)$ as in Lemma~\ref{L:optimal marking}.
\end{lemma}

\proof
Let $k\in\NN_0$ be fixed. Let $\varepsilon=\varepsilon(k) >0$ be a tolerance to be fixed later. Since $u\in \mathbb{A}_{s}$, there exist a mesh $\Tau_{\varepsilon}\in\mathbb{T}$ and a function $V_\varepsilon\in \VV_{\Tau_\varepsilon}$ such that
\[
\#\Tau_{\varepsilon }-\#\Tau_{0}\le |u|_{s}^{\frac{1}{s}}\varepsilon^{-\frac{1}{s}}\quad\text{ and }\quad
\|\nabla(V_\varepsilon-u)\|_\Omega^2+\osci_{\Tau_\varepsilon}^2(V_\varepsilon)\le \varepsilon^2.
\]

Let $\Taustar\definedas\Tau_{\varepsilon}\oplus \Tau_k$ the overlay of $\Tau_{\varepsilon}$ and $\Tau_k$ (cf. Lemma~\ref{L:overlay}). Since $V_\varepsilon\in\VV_\Taustar$, we have that $\osci_{\Tau_\varepsilon}(V_\varepsilon)\ge\osci_{\Taustar}(V_\varepsilon)$, and from Lemma~\ref{L:generalized Cea}, if $U_*\in\VV_\Taustar$ denotes the solution of the discrete problem~\eqref{E:disc prob} in $\VV_\Taustar$, we obtain that
\[
\|\nabla(U_*-u)\|_\Omega^2+\osci_{\Taustar}^2(U_*)\le2C_E\frac{C_A}{c_A}\left(\|\nabla(V_\varepsilon-u)\|_\Omega^2+\osci_{\Tau_\varepsilon}^2(V_\varepsilon)\right)\le 2C_E\frac{C_A}{c_A}\varepsilon^2.
\]
Let $\varepsilon$ be such that $$\|\nabla(U_*-u)\|_\Omega^2+\osci_{\Taustar}^2(U_*)\le\nu \left(\|\nabla(U_k-u)\|_\Omega^2+\osci_{\Tau_k}^2(U_k)\right)=2C_E\frac{C_A}{c_A}\varepsilon^2,$$ where $\nu$ is the constant given by Lemma~\ref{L:optimal marking}. Thus, this lemma yields
\[
\eta_{\Tau_k}(U_k;\mathcal{R}_k)\geq \theta\eta_{\Tau_k}(U_k),
\]
if $\mathcal{R}_k$ denotes the subset of $\Tau_k$ consisting of elements which were refined to get $\Taustar$. Taking into account that $\MM_k$ is a \emph{minimal} subset of $\Tau_k$ satisfying the D\"orfler's criterion, using Lemma~\ref{L:overlay} and recalling the choice of $\varepsilon$ we conclude that
\begin{align*}
\#\MM_k &\leq \#\mathcal{R}_k\leq \#\Taustar-\#\Tau_k \leq \#\Tau_{\varepsilon}-\#\Tau_{0}\le |u|_{s}^{\frac{1}{s}}\varepsilon^{-\frac{1}{s}}\\
&= \left(\frac{2C_EC_A}{\nu c_A}\right)^{\frac{1}{2s}}|u|_{s}^{\frac{1}{s}}\left(\|\nabla(U_k-u)\|_\Omega^2+\osci_{\Tau_k}^2(U_k)\right)^{-\frac{1}{2s}}.\tag*{\qedsymbol}
\end{align*}

The next result bounds the complexity of a mesh $\Tau_k$ in terms of the number of elements that were marked from the beginning of the iterative process, assuming that all the meshes were obtained by the bisection algorithm of~\cite{Stevenson-refine}, and that the initial mesh was properly labeled (satisfying condition (b) of Section 4 in~\cite{Stevenson-refine}).

\begin{lemma}[Complexity of \textsf{REFINE}]\label{L:stevenson}
Let us assume that $\Tau_0$ satisfies the labeling condition (b) of Section 4 in Ref.~\cite{Stevenson-refine}, and consider the sequence $\{\Tau_k\}_{k\in\NN_0}$ of refinements of $\Tau_0$ where $\Tau_{k+1}\definedas \textsf{REFINE}(\Tau_k,\MM_k,n)$ with $\MM_k\subset\Tau_k$. Then, there exists a constant $C_S>0$ solely depending on $\Tau_0$ and the number of refinements $\nref$ performed by \textsf{REFINE} to marked elements, such that
$$\#\Tau_k-\#\Tau_{0}\le C_S \sum_{i=0}^{k-1} \#\MM_i,\qquad\text{for all } k\in\NN.$$
\end{lemma}

The next result will use Lemma~\ref{L:stevenson} and is a consequence of the global lower bound (Theorem~\ref{T:cota inferior global}), the bound for the cardinality of $\MM_k$ given by Lemma~\ref{L:cardinality} and the contraction property of Theorem~\ref{T:propiedad de contraccion}. This is the second main result of the paper.

\begin{theorem}[Quasi-optimal convergence rate]\label{T:opt_main}
Let us assume that $\Tau_0$ satisfies the labeling condition (b) of Section 4 in Ref.~\cite{Stevenson-refine}. Let us assume that the weak solution $u$ of problem~\eqref{E:cont prob} belongs to $\mathbb{A}_{s}$. If $\{U_k\}_{k\in\NN_0}$ denotes the sequence computed through the Adaptive Algorithm, and the marking parameter $\theta$ satisfies $0<\theta<\theta_0$ (cf.~\eqref{E:restriccion theta}), then
\begin{equation}\label{E:main}
\big[\|\nabla(U_k-u)\|_\Omega^2+\osci_{\Tau_k}^2(U_k)\big]^{\frac12}\le C |u|_s(\#\Tau_k-\#\Tau_0)^{-s},\qquad\forall\,k\in\NN,
\end{equation}
where $C>0$ depends on $d$, $\kappa_\TT$, problem data, the number of refinements $n$ performed over each marked element, the marking parameter $\theta$, and the regularity index $s$.
\end{theorem}

\begin{proof}
Let $k\in\NN$ be fixed. The global lower bound (Theorem~\ref{T:cota inferior global}) yields
\begin{equation*}
\|\nabla(U_i-u)\|_\Omega^2+\mu
\eta_{\Tau_i}^2(U_i)\le \big(1+\mu C_L^{-1}\big)\big[\|\nabla(U_i-u)\|_\Omega^2+\osci_{\Tau_i}^2(U_i)\big],\qquad 0\le i\le k-1,
\end{equation*}
where $\mu$ is the constant appearing in Theorem~\ref{T:propiedad de contraccion}. Using Lemmas~\ref{L:stevenson} and~\ref{L:cardinality} it follows that
\begin{align} 
\#\Tau_k-\#\Tau_0&\le C_S \sum_{i=0}^{k-1} \#\MM_i
\le C_S \left(\frac{2C_EC_A}{\nu c_A}\right)^{\frac{1}{2s}}|u|_s^{\frac{1}{s}} \sum_{i=0}^{k-1}\big[\|\nabla(U_i-u)\|_\Omega^2+\osci_{\Tau_i}^2(U_i)\big]^{-\frac{1}{2s}}\nonumber\\
&\le C_S \left(\frac{2C_EC_A}{\nu c_A}\right)^{\frac{1}{2s}}
|u|_s^{\frac{1}{s}} \big(1+\mu C_L^{-1}\big)^{\frac{1}{2s}} \sum_{i=0}^{k-1}\big[\|\nabla(U_i-u)\|_\Omega^2
+\mu \eta_{\Tau_i}^2(U_i)\big]^{-\frac{1}{2s}}.
\label{E:aux1 casi optimalidad}
\end{align}
Since we do not have a contraction for the quantity $\big[\|\nabla(U_i-u)\|_\Omega^2+\mu \eta_{\Tau_i}^2(U_i)\big]$ as happens in the linear problem case, we now proceed as follows. 
We define $z_i^2\definedas [\FF(U_i)-\FF(u)]+\mu \eta_{\Tau_i}^2(U_i)$, the contraction property (Theorem~\ref{T:propiedad de contraccion}) yields $z_{i+1}\leq \rho z_{i}$ and thus, $z_{i}^{-\frac{1}{s}}\leq \rho^{\frac{1}{s}} z_{i+1}^{-\frac{1}{s}}$. Since $\rho<1$, taking into account~\eqref{E:equivalence for error}, we obtain that\footnote{In this estimation we assume for simplicity that $c_A$ and $C_A$ are chosen so that $c_A\le 2\le C_A$.}
\begin{align*}
\sum_{i=0}^{k-1}\big(\|\nabla(U_i-u)\|_\Omega^2+\mu
\eta_{\Tau_i}^2(U_i)\big)^{-\frac{1}{2s}}&\le (C_A/2)^{\frac{1}{2s}}\sum_{i=0}^{k-1}z_{i}^{-\frac{1}{s}}\le (C_A/2)^{\frac{1}{2s}}\sum_{i=1}^\infty (\rho^{\frac1{s}})^i z_k^{-\frac{1}{s}} \\ &=(C_A/2)^{\frac{1}{2s}}\frac{\rho^{\frac{1}{s}}}{1-\rho^{\frac{1}{s}}} z_k^{-\frac{1}{s}}\\
&\le (C_Ac_A^{-1})^{\frac{1}{2s}}\frac{\rho^{\frac{1}{s}}}{1-\rho^{\frac{1}{s}}}\big(\|\nabla(U_k-u)\|_\Omega^2+\mu
\eta_{\Tau_k}^2(U_k)\big) ^{-\frac{1}{2s}}.
\end{align*}
Using the last estimation in~\eqref{E:aux1 casi optimalidad}, it follows that
\[
\#\Tau_k-\#\Tau_{0}\le C_S \left(\frac{2C_EC_A}{\nu c_A}\right)^{\frac{1}{2s}}|u|_s^{\frac{1}{s}}\big(1+\mu C_L^{-1}\big)^{\frac{1}{2s}} 
(C_Ac_A^{-1})^{\frac{1}{2s}}\frac{\rho^{\frac{1}{s}}}{1-\rho^{\frac{1}{s}}}\left(\|\nabla(U_k-u)\|_\Omega^2+\mu \eta_{\Tau_k}^2(U_k)\right)
^{-\frac{1}{2s}},
\]
and using that $\gosck{U_k}\le \eta_{\Tau_k}(U_k)$ and raising to the $s$-power we have that
\[
(\#\Tau_k-\#\Tau_{0})^s \le \frac{C_S^sC_A}{c_A} \left(\frac{2C_E}{\nu }\right)^{\frac{1}{2}}\big(1+\mu C_L^{-1}\big)^{\frac{1}{2}} 
\frac{\rho}{(1-\rho^{\frac{1}{s}})^s}|u|_s\big(\|\nabla(U_k-u)\|_\Omega^2+\mu\osci_{\Tau_k}^2(U_k)\big)^{-\frac{1}{2}}.
\]
Finally, from this last estimation the assertion~\eqref{E:main} follows, and the proof is concluded.
\end{proof}

We conclude this article with a few remarks.

\begin{remark}
The problem given by~\eqref{E:cons law} is a particular case of the more general problem
\begin{equation*}
\left\{
\begin{aligned}
-\nabla\cdot \big[\alpha(\,\cdot\,,|\nabla u|_\AAA^2)\AAA\nabla u\big]&= f\qquad & & \text{in}\,\Omega\\
u&= 0\qquad & &\text{on}\,\partial \Omega,
\end{aligned}
\right.
\end{equation*}
where $\alpha:\Omega\times \RR_+\to \RR_+$ and $f\in L^2(\Omega)$ satisfy the properties assumed in the previous sections, and $\AAA:\Omega\to \RR^{d\times d}$ is such that $\AAA(x)$ is a symmetric matrix, for all $x\in\Omega$, and uniformly elliptic, i.e., there exist constants $\underline{a},\overline{a}>0$ such that
\begin{equation*}
\underline{a}|\xi|^2\leq \AAA(x)\xi\cdot \xi\leq \overline{a}|\xi|^2,\qquad \forall~x\in\Omega,\,\xi\in\RR^d.
\end{equation*}
If $\AAA$ is piecewise constant over an initial conforming mesh $\Tau_0$ of $\Omega$, then the convergence and optimality results previously presented also hold for this problem.
\end{remark}

\begin{remark}\label{R:polynomial degree}
We have assumed the use of \emph{linear} finite elements for the discretization (see~\eqref{E:V_Tau}). It is important to notice that the only place where we used this is for proving~\eqref{E:bound for gtau}. The rest of the steps of the proof hold regardless of the degree of the finite element space. The use of linear finite elements is customary in nonlinear problems, because they greatly simplify the analysis.
\end{remark}

{\setlength{\parindent}{0pt}
\small

\medskip
\rule{.5\textwidth}{1.5pt}

\medskip
\textbf{Affiliations}
\begin{description}
\item[Eduardo M.\ Garau:] Consejo Nacional de Investigaciones Cient\'{\i}ficas y T\'{e}cnicas and Universidad Nacional del Litoral, Argentina, \url{egarau@santafe-conicet.gov.ar}. 

Partially supported by CONICET (Argentina) through Grant PIP 112-200801-02182, and Universidad Nacional del Litoral through Grant CAI+D PI 062-312.

Address: IMAL, G\"uemes 3450. S3000GLN Santa Fe, Argentina.

\item[Pedro Morin:] Consejo Nacional de Investigaciones Cient\'{\i}ficas y T\'{e}cnicas and Universidad Nacional del Litoral, Argentina, \url{pmorin@santafe-conicet.gov.ar}.

 Partially supported by CONICET (Argentina) through Grant PIP 112-200801-02182, and Universidad Nacional del Litoral through Grant CAI+D PI 062-312.

Address: IMAL, G\"uemes 3450. S3000GLN Santa Fe, Argentina.

\item[Carlos Zuppa:] Universidad Nacional de San Luis, Argentina,
\url{zuppa@unsl.edu.ar}.
 
Partially supported by Universidad Nacional de San Luis through Grant 22/F730-FCFMyN.

Address: Departamento de Matem\'atica, Facultad de Ciencias F\'\i sico Ma\-te\-m\'a\-ticas y Naturales, Universidad Nacional de San Luis, Chacabuco 918, D5700BWT San Luis, Argentina.
\end{description}
}

\end{document}